\documentclass[reqno,10pt]{amsart}

\usepackage[left=3cm,right=3cm,top=3cm,bottom=3cm]{geometry}
\usepackage{amssymb,bm}
\usepackage{mathrsfs}
\usepackage{xcolor}
\usepackage{enumitem}
\usepackage{appendix}
\usepackage{diffcoeff}
\usepackage{amsmath,cases}

\usepackage{graphicx}
\usepackage{subfig}

\newtheorem{theorem}{Theorem}[section]
\newtheorem{lemma}[theorem]{Lemma}
\newtheorem{proposition}[theorem]{Proposition}
\newtheorem{corollary}[theorem]{Corollary}

\theoremstyle{definition}

\theoremstyle{remark}
\newtheorem{remark}[theorem]{Remark}

\numberwithin{equation}{section}

\DeclareMathOperator{\dive}{div}

\begin{document}

\title[Two-parameter relaxation of the incompressible Navier-Stokes equations]{A hyperbolic relaxation system of the incompressible Navier-Stokes equations with artificial compressibility}



\author[Qian Huang]{Qian Huang*}
\address{Institute of Applied Analysis and Numerical Simulation, University of Stuttgart, 70569 Stuttgart, Germany}
\email{qian.huang@mathematik.uni-stuttgart.de; hqqh91@qq.com}
\thanks{* Corresponding author}

\author[Christian Rohde]{Christian Rohde}
\address{Institute of Applied Analysis and Numerical Simulation, University of Stuttgart, 70569 Stuttgart, Germany}
\email{christian.rohde@mathematik.uni-stuttgart.de}

\author[Wen-An Yong]{Wen-An Yong}
\address{Department of Mathematical Sciences, Tsinghua University, Beijing 100084, China; Beijing Institute of
Mathematical Sciences and Applications, Beijing 101408, China}
\email{wayong@tsinghua.edu.cn}

\author[Ruixi Zhang]{Ruixi Zhang}
\address{Department of Mathematical Sciences, Tsinghua University, Beijing 100084, China}
\email{zhangrx24@mails.tsinghua.edu.cn; 1553548358@qq.com}

\subjclass[2020]{Primary 35Q30 · 76D05}

\keywords{Incompressible Navier-Stokes equations, relaxation approximations, artificial compressibility, hyperbolic singular perturbations, energy estimates}

\date{}

\dedicatory{}

\begin{abstract}

We introduce a new hyperbolic approximation to the incompressible Navier-Stokes equations by incorporating a first-order relaxation and using the artificial compressibility method. With two relaxation parameters in the model, we rigorously prove the asymptotic limit of the system towards the incompressible Navier-Stokes equations as both parameters tend to zero. Notably, the convergence of the  approximate pressure variable is achieved by the help of a linear `auxiliary' system and energy-type error estimates of its differences with the two-parameter model and the Navier-Stokes equations.

\end{abstract}

\maketitle

\section{Introduction}

Consider the incompressible Navier-Stokes equations
\begin{equation} \label{eq:ns}
\left\{ \
\begin{aligned}
  &\nabla\cdot u^{NS} = 0, \\
  &\partial_t u^{NS} + \nabla \cdot (u^{NS} \otimes u^{NS}) = -\nabla p^{NS} + \Delta u^{NS}, \\
  &u^{NS}(0,\cdot) =u_0^{NS}
\end{aligned}
\right.
\end{equation}
for $(t,x)\in [0,T]\times \mathbb T^2$, where $\mathbb T^2$ is the unit periodic square $\mathbb R^2 / \mathbb Z^2$. This set of partial differential equations describes the motion of viscous flows, with the unknowns $(p^{NS},u^{NS}):\ (t,x) \mapsto \mathbb R\times \mathbb R^2$ being respectively the pressure and velocity of the fluid. Here we set the viscosity coefficient as unity. Since only the gradient of $p^{NS}$ is involved, it is convenient to assume $\int_{\mathbb T^2} p^{NS}(t,x)dx = 0$ for all $t$. The existence of smooth solutions for $(p^{NS},u^{NS})$ is well known for smooth initial data \cite{ladyzhenskaya1969,kato1972}.
It is helpful for our purpose to introduce $U^{NS} = u^{NS}\otimes u^{NS} -\nabla u^{NS} \in \mathbb R^{2\times 2}$ and rewrite the $u^{NS}$-equation as
\[
  \partial_t u^{NS} + \nabla \cdot U^{NS} + \nabla p^{NS} = 0.
\]

There are many efforts to approximate (\ref{eq:ns}). An important direction is inspired by the fact that (\ref{eq:ns}) is the hydrodynamic limit of the Boltzmann equation \cite{de1989CPAM,golse2004}. In this regard, it has been rigorously justified that (\ref{eq:ns}) is the limit system of several velocity-discrete Boltzmann-BGK (Bhatnagar-Gross-Krook) equations under the diffusive scaling \cite{JY2003,bianchini2019}. Indeed, such relaxation approximations have been extensively studied in different kinds of systems to benefit both analytical and numerical treatments \cite{jinxin1995,aregba2000}. 
In the spirit of local relaxation approximations, Brenier et al. \cite{Brenier2003} introduced a first-order model via the diffusive scaling as
\begin{equation} \label{eq:relax_para}
\left\{ \
\begin{aligned}
  &\nabla\cdot u^\delta = 0, \\
  &\partial_t u^\delta + \nabla \cdot U^\delta + \nabla p^\delta = 0, \\
  &\delta \partial_t U^\delta + \nabla u^\delta = u^\delta\otimes u^\delta - U^\delta, \\
  &u^\delta(0,\cdot)=u_0^\delta, \ U^\delta(0,\cdot) = U_0^\delta,
\end{aligned}
\right.
\end{equation}
with $\delta>0$ and $U^\delta: (t,x) \mapsto \mathbb R^{2\times 2}$. It is then proved in \cite{Brenier2003} that, for appropriate initial data, a smooth solution $u^\delta$ exists globally and converges to a solution $u^{NS}$ of (\ref{eq:ns}) in $L^\infty([0,T],L^2(\mathbb T^2))$ as $\delta\to 0$, with the error bounded by $O(\sqrt\delta)$. These results have then been advanced in \cite{paicu2007} and further in \cite{hachicha2014}. For more analysis on this system we refer to \cite{racke2012,ilyin2018}.
The semi-linear, first-order structure of (\ref{eq:relax_para}) could bring advantages to the numerical schemes and may even be exploited to compute simpler descriptions for the statistics of turbulent flows \cite{Lundgren,friedrich2012}.

Another influential approach to approximate (\ref{eq:ns}) is the artificial compressibility (AC) method introduced by Chorin \cite{chorin1967ams,chorin1967jcp,chorin1968} and Temam \cite{temam1969I,temam1969II}. This method uses an evolution equation $\epsilon \partial_t p^\epsilon + \nabla \cdot u^\epsilon=0$ (with $\epsilon>0$ and $\sqrt \epsilon$ called the artificial Mach number) to replace the original divergence-free constraint $\nabla\cdot u^{NS}=0$ that raises major computational difficulties. Then the steady-state solution of (\ref{eq:ns}) can be numerically obtained by computing the large-time limit of the solution of the perturbed, hyperbolic-parabolic system. This approach has a profound impact on the development of robust and efficient numerical schemes for (\ref{eq:ns}) \cite{gresho1991,nithiarasu2003,massa2022}.
For the analysis of the singular limit $\epsilon\to 0$, a classical result on the convergence of the solution $(u^\epsilon,\nabla p^\epsilon)$ of the perturbed system towards $(u^{NS},\nabla p^{NS})$ of (\ref{eq:ns}) is presented in \cite{temam2001ns}. For more analysis on the system we refer to \cite{donatelli2010,kagei2021}.

Combining these ideas, we introduce a new approximation to (\ref{eq:ns}) on $[0,T_{\epsilon,\delta})\times \mathbb T^2$ as:
\begin{equation} \label{eq:relax}
\left \{ \
\begin{aligned}
  &\epsilon \partial_t p^{\epsilon,\delta} + \nabla \cdot u^{\epsilon,\delta} =0, \\
  &\partial_t u^{\epsilon,\delta} + \nabla \cdot U^{\epsilon,\delta} + \nabla p^{\epsilon,\delta} = 0, \\
  &\delta \partial_t U^{\epsilon,\delta} + \nabla u^{\epsilon,\delta} = u^{\epsilon,\delta}\otimes u^{\epsilon,\delta} - U^{\epsilon,\delta}, \\
  &u^{\epsilon,\delta}(0,\cdot)=u^{\epsilon,\delta}_0, \ p^{\epsilon,\delta}(0,\cdot)=p^{\epsilon,\delta}_0, \ U^{\epsilon,\delta}(0,\cdot)=U^{\epsilon,\delta}_0,
\end{aligned}
\right.
\end{equation}
where $\epsilon,\delta>0$ are two relaxation parameters. Although this seems to be a natural extension of the above-mentioned models, to our best knowledge, (\ref{eq:relax}) remains yet to be investigated. Now that (\ref{eq:relax}) is a hyperbolic system which admits only finite disturbance propagation velocities and therefore is more physical. 
It surely has the potential to inherit the rich numerical methods for hyperbolic conservation systems. The existence of the solutions to (\ref{eq:relax}) is guaranteed by the well-established theory \cite{kato1975}. At the singular limit $\epsilon,\delta\to 0$, (\ref{eq:ns}) is formally recovered from (\ref{eq:relax}). The purpose of this paper is hence to present a rigorous analysis on the convergence $(u^{\epsilon,\delta},p^{\epsilon,\delta})\to (u^{NS},p^{NS})$ as $\epsilon,\delta\to 0$.

Let us remark that, while the one-parameter limit problem has been widely studied for many equations \cite{klainerman1981,Yong1999}, 
such two-parameter limit problems for (\ref{eq:relax}) are far less explored. 
Considering the limit problems for relaxation approximations and artificial compressibility, it may seem natural to pass to the limit `in sequence' as $(p,u)^{\epsilon,\delta}\to (p,u)^\epsilon \to (p,u)^{NS}$ or $(p,u)^{\epsilon,\delta}\to (p,u)^\delta \to (p,u)^{NS}$. However, the existing convergence results for (\ref{eq:relax_para}) do not characterize the behavior of $p^\delta$ \cite{Brenier2003,paicu2007,hachicha2014}, whereas those for AC \cite{temam2001ns} do not give error estimates which are essential for our purpose.
Therefore, these approaches seem challenging to pursue.
For the two-parameter limit problems, we mention a general convergence-stability principle developed in \cite{Yong2001,brenier2005} for initial value problems of quasilinear symmetrizable hyperbolic systems depending singularly on parameters in a topological space. With this principle and the energy method, the convergence of a relaxed hydrodynamic model for semiconductors is proved in \cite{xu2009}, where two relaxation times are involved.

In our work, as the model (\ref{eq:relax}) is a symmetrizable hyperbolic system for $(\sqrt\epsilon p^{\epsilon,\delta}, u^{\epsilon,\delta}, \sqrt\delta U^{\epsilon,\delta})$, the convergence $u^{\epsilon,\delta}\to u^{NS}$ can be obtained (by assuming appropriate initial data) with the standard energy estimates and the Gagliardo-Nirenberg inequalities. But this straightforward method seems difficult to reach a convergence for $p^{\epsilon,\delta}\to p^{NS}$ (due to the existence of $\sqrt\epsilon$). To overcome the barrier, we introduce an `intermediate' linear system to pass to the limit as $(p,u)^{\epsilon,\delta}\to (p,u)^{\text{linear}} \to (p,u)^{NS}$ with energy-type error estimates in both steps. This strategy works successfully for $\delta = O(\sqrt\epsilon)$. In addition, the inclusion of a linear friction (relaxation) term in the model is straightforward.

The remainder of the paper is organized as follows. The main results and lemmas are collected in Section \ref{sec:main}. Some lemmas are proved in the appendices. Section \ref{sec:uconv} presents the proofs of Theorems \ref{thm:uconvL2} and \ref{thm:uconvH1} (the convergence of $u^{\epsilon,\delta}$), and Section \ref{sec:upconv} is devoted to the proof of Theorem \ref{thm:pconv} (the convergence of $(p,u)^{\epsilon,\delta}$). Conclusions are drawn in Section \ref{sec:concl}.

We fix some notations. Throughout the text, $C$ and $C'$ denote sufficiently large, generic constants that may rely on the initial datum of (\ref{eq:relax}) and (\ref{eq:ns}) but \textit{are independent of} $(\epsilon,\delta)$. Sometimes $C_T$, $C_t$ or $C_\gamma$ are used to express the dependence on the parameters $T$, $t$ (time) or $\gamma$. The $L^p$- ($p=2$ or $\infty$) and $H^s$- ($s=1,2$) norms are both for $\mathbb T^2$. We denote by $A:B$ the Frobenius inner product of two matrices $A$ and $B$, $\dive = (\nabla \cdot)$ the divergence operator, and $\Delta=\nabla^2$ the Laplacian operator.

\section{Main results} \label{sec:main}
In this section, we present our main results. In what follows we assume that $(p^{NS},u^{NS})$ is the smooth solution of (\ref{eq:ns}), and that $(u^{\epsilon,\delta}, p^{\epsilon,\delta}, U^{\epsilon,\delta})$ is a series of solutions of (\ref{eq:relax}) on $(t,x)\in [0,T_{\epsilon,\delta})\times \mathbb T^2$ with smooth initial data $(u_0^{\epsilon,\delta}, p_0^{\epsilon,\delta}, U_0^{\epsilon,\delta})$. Notice again that the existence and uniqueness of the solutions to (\ref{eq:relax}) are guaranteed by the classical theory \cite{kato1975}. For the asymptotic behavior of the sequence $u^{\epsilon,\delta}$ as $\epsilon,\delta\to 0$, we have


\begin{theorem} \label{thm:uconvL2}
 Assume that there exists a constant $a\in [1,2)$ such that
  \[
    \|u_0^{\epsilon,\delta}-u^{NS}_0\|_{L^2}^2 + (\epsilon+\delta)^a \| \Delta u_0^{\epsilon,\delta} \|_{L^2}^2 \le C(\epsilon+\delta)
  \]
  and
  \[
    \| \left( p_0^{\epsilon,\delta}, U_0^{\epsilon,\delta} \right) \|_{L^2}^2 + (\epsilon+\delta)^a \| \Delta \left( p_0^{\epsilon,\delta}, U_0^{\epsilon,\delta} \right) \|_{L^2}^2 \le C.
  \]
  Then the existing time $T_{\epsilon,\delta}$ of the solutions $(u^{\epsilon,\delta}, p^{\epsilon,\delta}, U^{\epsilon,\delta})$ satisfies
  \[
    T_{\epsilon,\delta} \ge C' \left[\left( \frac{a}{2}-1 \right) \log(\epsilon+\delta) - 1\right].
  \]
  Moreover, for any $T<T_{\epsilon,\delta}$, we have
  \[
    \sup_{0<t<T} \| u^{\epsilon,\delta}(t,\cdot) - u^{NS}(t,\cdot) \|_{L^2}^2 \le C_T(\epsilon+\delta)
  \]
  and
  \[
    \sup_{0<t<T} \| u^{\epsilon,\delta}(t,\cdot) \|_{L^\infty} \le C_T
  \]
  for $\epsilon+\delta\ll 1$.
\end{theorem}

This theorem indicates that the existing time $T_{\epsilon,\delta}$ is uniformly bounded from below and $T_{\epsilon,\delta}\to\infty$ as $\epsilon,\delta\to 0$. 
A similar result can be stated in the $H^1$-framework as
\begin{theorem} \label{thm:uconvH1}
  Assume that there exists a constant $a\ge 1$ such that
  \[
    \| u_0^{\epsilon,\delta} - u_0^{NS} \|_{H^1}^2 + (\epsilon+\delta)^a \| \Delta u_0^{\epsilon,\delta}\|_{L^2}^2 \le C(\epsilon+\delta)
  \]
  and
  \[
    \| \left( p_0^{\epsilon,\delta}, U_0^{\epsilon,\delta} \right) \|_{H^1}^2 + (\epsilon+\delta)^a \| \Delta \left( p_0^{\epsilon,\delta}, U_0^{\epsilon,\delta} \right) \|_{L^2}^2 \le C.
  \]
  Likewise, for any $T<T_{\epsilon,\delta}$ (which still tends to $\infty$ as $\epsilon+\delta\to 0$), we have
  \[
    \sup_{0<t<T} \| u^{\epsilon,\delta}(t,\cdot) - u^{NS}(t,\cdot) \|_{L^2}^2 \le C_T (\epsilon+\delta)
  \]
  and
  \[
    \sup_{0<t<T} \| u^{\epsilon,\delta}(t,\cdot) \|_{L^\infty} \le C_T
  \]
  for $\epsilon+\delta\ll 1$.
\end{theorem}

\begin{remark}
  In the single-parameter limit $\delta\to 0$ of (\ref{eq:relax_para}), a similar convergence rate $\sup_{0<t<T}\|u^\delta(t,\cdot)-u^{NS}(t,\cdot)\|_{L^2}^2 \le C_T\sqrt\delta$ is proved in \cite{Brenier2003} for appropriate initial conditions.
\end{remark}

It will become clear in Section \ref{sec:uconv} that the proof by standard energy estimates for the above theorems only gives $\sup_{0<t<T} \| p^{\epsilon,\delta}(t,\cdot) - p^{NS}(t,\cdot) \|_{L^2}^2 = O(1)$. Instead, a convergence result for $p^{\epsilon,\delta}$ can be obtained for $
\delta=O(\sqrt\epsilon)$ and stricter requirements on the initial data:
\begin{theorem} \label{thm:pconv}
  Assume $\delta \le C\sqrt\epsilon$ as $\epsilon\to 0$. Suppose that
  \[
    \|u_0^{\epsilon,\delta}-u_0^{NS}\|_{H^1} + \sqrt\epsilon \|p_0^{\epsilon,\delta}-p_0^{NS}\|_{H^1} + (\epsilon+\delta)\|U_0^{\epsilon,\delta}\|_{H^1} \le C(\epsilon+\delta)
  \]
  and
  \[
    \|\nabla\dive u_0^{\epsilon,\delta}\|_{L^2} + \delta \left( \| \Delta p_0^{\epsilon,\delta} \|_{L^2} + \| \nabla \dive U_0^{\epsilon,\delta} \|_{L^2} \right) \le C(\epsilon+\delta).
  \]
  Moreover, suppose that for some $a>0$, it holds that
  \[
    \| \Delta (u_0^{\epsilon,\delta}, p_0^{\epsilon,\delta}, U_0^{\epsilon,\delta}) \|_{L^2} \le C(\epsilon+\delta)^{-a}.
  \]
  Then for any $T<T_{\epsilon,\delta}$, we have
  \[
    \sup_{0<t<T} \left( \|u^{\epsilon,\delta}(t,\cdot)-u^{NS}(t,\cdot)\|_{H^1} + \sqrt\epsilon \| p^{\epsilon,\delta}(t,\cdot) - p^{NS}(t,\cdot) \|_{H^1} \right) \le C_T(\epsilon+\delta)
  \]
  for $\epsilon + \delta \ll 1$.
\end{theorem}

This immediately leads to a convergence result of $p^{\epsilon,\delta}$:
\begin{corollary}
  Assume $\delta=o(\sqrt\epsilon)$ (that is, $\frac{\delta}{\sqrt\epsilon}\to 0$ as $\epsilon \to 0$). Then with the same assumptions as stated in Theorem \ref{thm:pconv}, we have
  \[ \sup_{0<t<T}\|p^{\epsilon,\delta}(t,\cdot) - p^{NS}(t,\cdot)\|_{H^1} \le C_T\left(\sqrt\epsilon + \frac{\delta}{\sqrt\epsilon}\right) \to 0 \]
  as $\epsilon\to 0$ for $T<T_{\epsilon,\delta}$.
\end{corollary}

It is straightforward to include a linear friction (relaxation) term in the model:
\begin{proposition}
  All the above results hold if (\ref{eq:ns}) contains a linear friction term as $\partial_t u^{NS} + \nabla \cdot (u^{NS}\otimes u^{NS}) = -\nabla p^{NS} + \Delta u^{NS} - u^{NS}$ and (\ref{eq:relax}) becomes $\partial_t u^{\epsilon,\delta} + \nabla \cdot U^{\epsilon,\delta} + \nabla p^{\epsilon,\delta} = -u^{\epsilon,\delta}$ accordingly.
\end{proposition}
To verify Theorems \ref{thm:uconvL2} and \ref{thm:uconvH1} for this case follows trivially from the arguments in Section \ref{sec:uconv} since the extra term dissipates energy. Theorem \ref{thm:pconv} for this case is justified in Remark \ref{rm:fric}.

A set of auxiliary results will be needed for the proof of the theorems. First, the variants of the Gagliardo-Nirenberg inequality:
\begin{lemma} \label{lem:GNineq}
  For $v\in H^2(\mathbb T^2)$, it holds that
  \begin{subequations}
    \begin{align}
      \|v\|_{L^4}^2 &\le C \|v\|_{L^2} \|\nabla v\|_{L^2} \quad \text{(Ladyzhenskaya's inequality)}, \label{eq:ladyzhen} \\
      \|\nabla v\|_{L^2} &\le \|v\|_{L^2}^{1/2} \|\nabla^2 v\|_{L^2}^{1/2}, \label{eq:est_dv} \\
      \|v\|_{L^\infty}^2 &\le C \|v\|_{L^2}^\gamma \|\nabla v\|_{L^2}^{2-2\gamma} \|\nabla^2 v\|_{L^2}^\gamma + C\|v\|_{L^2}^2, \quad \forall \ 0<\gamma\le 1. \label{eq:est_inf}
    \end{align}
    \end{subequations}
\end{lemma}
For completeness, the proofs of (\ref{eq:est_dv}) and (\ref{eq:est_inf}) are given in Appendix \ref{app:GNineq}, whereas the proof of (\ref{eq:ladyzhen}) can be found in \cite{ladyzhenskaya1969}. Remark that these interpolation inequalities have different forms on $\mathbb T^3$, and in this work we restrict ourselves to two-dimensional cases.

Second, the following blow-up criterion will be used.
\begin{lemma} \label{lem:explosion}
  Let $U\in C^0([0,T),H^2(\mathbb T^2)) \cap C^1([0,T),H^1(\mathbb T^2))$ be a classical solution of
  \begin{equation} \label{eq:quad_model}
   \partial_t U + \sum_{i=1}^{2} A_i \partial_{x_i} U = Q(U) + F(t,x)
  \end{equation}
  with
  \[ F\in L_{loc}^1(\mathbb R_+,H^2(\mathbb T^2)) \cap C^0(\mathbb R_+,H^1(\mathbb T^2)), \]
$A_i=A_i^T$ being constant matrices, and $Q(U)$ being a quadratic function of $U$ with $Q(0)=0$. If the existing time $T < \infty$, then we have
  \[
    \int_0^T \|U(t)\|_{L^\infty} dt = \infty.
  \]
\end{lemma}

The proof is presented in Appendix \ref{app:blow}. Let us remark that the existence of such $U$ as the classical solution of (\ref{eq:quad_model}) is ensured in \cite{kato1975}. Also note that the blow-up caused by the nonlinear source $Q(U)$ here differs from that by the quasilinear coefficient (see, e.g., Theorem 4.16 of \cite{bahouri2011}).

\begin{lemma} \label{lem:helmholtz}
  For the two-dimensional velocity field $u\in H^1(\mathbb T^2)$, we have the Helmholtz decomposition
  \[ \|\nabla u\|_{L^2}^2 = \|\nabla\cdot u\|_{L^2}^2 + \|\nabla\times u\|_{L^2}^2 \]
\end{lemma}

\begin{proof}
  The result follows from a direct calculation as
  \[
  \begin{aligned}
    \|\nabla\cdot u\|_{L^2}^2 &= \int_{\mathbb T^2} (\partial_{x_1}u_1 + \partial_{x_2}u_2)^2 dx = \int_{\mathbb T^2} \left[(\partial_{x_1}u_1)^2 + (\partial_{x_2}u_2)^2 \right]dx + 2\int_{\mathbb T^2} (\partial_{x_1}u_1)(\partial_{x_2}u_2)dx,\\
    \|\nabla\times u\|_{L^2}^2&= \int_{\mathbb T^2} (\partial_{x_1}u_2 - \partial_{x_2}u_1)^2 dx = \int_{\mathbb T^2} \left[(\partial_{x_1}u_2)^2 + (\partial_{x_2}u_1)^2 \right]dx - 2\int_{\mathbb T^2} (\partial_{x_1}u_2)(\partial_{x_2}u_1)dx,
  \end{aligned}
  \]
  and
  \[
    \int_{\mathbb T^2} (\partial_{x_1}u_1)(\partial_{x_2}u_2)dx
    = - \int_{\mathbb T^2} u_1(\partial_{x_1}\partial_{x_2}u_2)dx
    = \int_{\mathbb T^2} (\partial_{x_2}u_1)(\partial_{x_1}u_2)dx.
  \]
\end{proof}

\section{Convergence of $u^{\epsilon,\delta}$} \label{sec:uconv}
To prove Theorems \ref{thm:uconvL2} and \ref{thm:uconvH1}, we denote by
\begin{equation} \label{eq:resvar}
  v = u^{\epsilon,\delta}-u^{NS}, \quad q = \sqrt{\epsilon}(p^{\epsilon,\delta}-p^{NS}), \quad V = \sqrt{\delta} (U^{\epsilon,\delta}-U^{NS})
\end{equation}
the smooth `residual' function on $(t,x)\in [0,T_{\epsilon,\delta})\times \mathbb T^2$.
The system for $W=(q,v,V)$ reads as
\begin{equation} \label{eq:res}
\left \{ \
\begin{aligned}
  &\partial_t q + \frac{1}{\sqrt\epsilon} \nabla \cdot v = -\sqrt \epsilon \partial_t p^{NS}, \\
  &\partial_t v + \frac{1}{\sqrt\delta} \nabla \cdot V + \frac{1}{\sqrt\epsilon} \nabla q = 0, \\
  &\partial_t V + \frac{1}{\sqrt\delta}\nabla v = \frac{2u^{NS} v + v\otimes v}{\sqrt\delta} - \frac{V}{\delta} - \sqrt \delta \partial_t U^{NS}
\end{aligned}
\right.
\end{equation}
with $u^{NS}v := (u^{NS}\otimes v+v\otimes u^{NS})/2$.
This is a symmetric hyperbolic balance system, for which we perform energy estimates based on the assumptions of the initial conditions.

\begin{proof}[Proof of Theorem \ref{thm:uconvL2}]
We define an energy as
\begin{equation} \label{eq:energy}
  E(t) = \| W(t,\cdot) \|_{L^2}^2 + (\epsilon+\delta)^a \| \Delta W(t,\cdot) \|_{L^2}^2 \quad \text{with} \quad 1\le a<2
\end{equation}
for the system (\ref{eq:res}). The assumptions in the theorem imply that
\begin{equation} \label{eq:ic}
  E(0)\le C(\epsilon+\delta)
\end{equation}
as $(\epsilon+\delta)^a \|\Delta v(0,\cdot)\|_{L^2}^2 \le 2(\epsilon+\delta)^a \|\Delta u_0^{\epsilon,\delta}\|_{L^2}^2 + 2(\epsilon+\delta)^a \|\Delta u_0^{NS}\|_{L^2}^2 \le C(\epsilon+\delta)$
for $a\ge 1$.

Then we have
\[
  \frac{1}{2}\frac{d}{dt}\|W\|_{L^2}^2 = -\sqrt\epsilon \int_{\mathbb T^2} q \partial_t p^{NS} dx - \sqrt{\delta} \int_{\mathbb T^2} V:\partial_t U^{NS} dx + \int_{\mathbb T^2} V:\left( \frac{2u^{NS} v + v\otimes v}{\sqrt\delta} - \frac{V}{\delta} \right)dx.
\]
Using H{\"o}lder's inequality and the inequality that
\[
  V:\frac{2u^{NS} v + v\otimes v}{\sqrt\delta} \le \frac{|V|^2}{4\delta} + |2u^{NS} v + v\otimes v|^2
\]
 (here $|A|^2 = A_{ij}A_{ij}$ with Einstein's summation hereinafter), we get
\[
\begin{aligned}
  \frac{d}{dt}\|W\|_{L^2}^2 &\le C\left( \sqrt\epsilon \|q\|_{L^2} + \sqrt\delta \|V\|_{L^2} \right) + 2\|2u^{NS} v + v\otimes v \|_{L^2}^2 - \frac{3}{2\delta} \|V\|_{L^2}^2 \\
  &\le  C\left( \sqrt\epsilon \|q\|_{L^2} + \sqrt\delta \|V\|_{L^2} + \|v\|_{L^2}^2 + \|v\|_{L^4}^4 \right) - \frac{3}{2\delta} \|V\|_{L^2}^2 \\
  &\le C \left( \sqrt\epsilon \|q\|_{L^2} + \sqrt\delta \|V\|_{L^2} + \|v\|_{L^2}^2 + \|v\|_{L^2}^3 \|\nabla^2v\|_{L^2} \right).
\end{aligned}
\]

Similarly, we have
\[
\begin{aligned}
  \frac{d}{dt}\|\Delta W\|_{L^2}^2 &\le C\left( \sqrt\epsilon \|\nabla^2 q\|_{L^2} + \sqrt\delta \|\nabla^2 V\|_{L^2} + \|\nabla^2(u^{NS} v)\|_{L^2}^2 + \|\nabla^2(v\otimes v)\|_{L^2}^2 \right) - \frac{3}{2\delta}\|\nabla^2 V\|_{L^2}^2 \\
  &\le C \left( \sqrt\epsilon \|\nabla^2 q\|_{L^2} + \sqrt\delta \|\nabla^2 V\|_{L^2} + \|v\|_{H^2}^2 + \|v\|_{L^2} \|\nabla^2 v\|_{L^2}^3 + \|v\|_{L^2}^3 \|\nabla^2v\|_{L^2} \right),
\end{aligned}
\]
where the nonlinear term
\[
  \|\nabla^2(v\otimes v)\|_{L^2}^2 = \int_{\mathbb T^2} \left( |v|^2 |\nabla^2 v|^2 + (v\cdot\nabla^2 v)^2 + (\nabla v_i \cdot \nabla v_j)(v_i\nabla^2 v_j + v_j\nabla^2 v_i + \nabla v_i \cdot \nabla v_j)\right) dx
\]
is treated, by repeatedly using H{\"o}lder's inequality and Lemma \ref{lem:GNineq} (with $\gamma=1$), as
\[
\begin{aligned}
  \int_{\mathbb T^2} \left( |v|^2 |\nabla^2 v|^2 + (v\cdot\nabla^2 v)^2 \right) dx
  &\le C \|v\|_{L^\infty}^2 \|\nabla^2 v\|_{L^2}^2 \\
  &\le C \left( \|v\|_{L^2} \|\nabla^2 v\|_{L^2}^3 + \|v\|_{L^2}^2 \|\nabla^2 v\|_{L^2}^2 \right) \\
  &\le C \left( \|v\|_{L^2} \|\nabla^2 v\|_{L^2}^3 + \|v\|_{L^2}^3 \|\nabla^2v\|_{L^2} \right), \\
  \int_{\mathbb T^2} (\nabla v_i \cdot \nabla v_j)(\nabla v_i \cdot \nabla v_j) dx & \le C\|\nabla v\|_{L^4}^4 \le C\| \nabla v \|_{L^2}^2 \|\nabla^2 v\|_{L^2}^2 \le C \|v\|_{L^2} \|\nabla^2 v\|_{L^2}^3, \\
  \int_{\mathbb T^2} (\nabla v_i \cdot \nabla v_j)(v_i\nabla^2v_j) dx &\le C\|\nabla v\|_{L^4}^2 \|v\|_{L^\infty} \|\nabla^2v\|_{L^2} \le C \|\nabla v\|_{L^2} \|v\|_{L^\infty} \|\nabla^2v\|_{L^2}^2.
\end{aligned}
\]
Notice that $\|v\|_{H^2}$ can be controlled by $\|v\|_{L^2}$ and $\|\nabla^2 v\|_{L^2}$ thanks to (\ref{eq:est_dv}). Then the above estimations lead to
\[
\begin{aligned}
  \frac{d}{dt}E(t) \le& C\sqrt \epsilon \left( \|q\|_{L^2} + (\epsilon+\delta)^a \|\nabla^2 q\|_{L^2} \right) + C\sqrt \delta \left( \|V\|_{L^2} +(\epsilon+\delta)^a\|\nabla^2 V\|_{L^2} \right) \\
  &+ C\left( \|v\|_{L^2}^2 + (\epsilon+\delta)^a \|\nabla^2 v\|_{L^2}^2 \right) + C\|v\|_{L^2}\|\nabla^2 v\|_{L^2}\left( \|v\|_{L^2}^2 + (\epsilon+\delta)^a \|\nabla^2 v\|_{L^2}^2 \right), 
\end{aligned}
\]
If $\epsilon+\delta\ll 1$, the first two terms can be controlled by $\sqrt{(\epsilon+\delta)E(t)}$ and the third term by $E(t)$. The fourth term can be controlled by $(\epsilon+\delta)^{-a/2}E(t)^2$ since
\[
  E(t) = \|W\|_{L^2}^2 + (\epsilon+\delta)^a \|\Delta W\|_{L^2}^2 \ge 2(\epsilon+\delta)^{\frac{a}{2}} \|W\|_{L^2} \|\Delta W\|_{L^2}.
\]
Therefore, the energy is bounded as
\begin{equation} \label{eq:enesL2}
  \frac{d}{dt}E(t) \le C\left( \sqrt{(\epsilon+\delta)E} + E + (\epsilon+\delta)^{-\frac{a}{2}} E^2 \right).
\end{equation}
With the bound on the initial energy in (\ref{eq:ic}), we see that
\[
  C't \ge \int_{C(\epsilon+\delta)}^{E(t)} \frac{dy}{\sqrt{(\epsilon+\delta)y}+y+(\epsilon+\delta)^{-\frac{a}{2}}y^2} \ge \int_{C(\epsilon+\delta)}^{E(t)} \frac{dy}{(C^{-\frac{1}{2}}+1) y + (\epsilon+\delta)^{-\frac{a}{2}} y^2}.
\]
Since $(\epsilon+\delta)^{-\frac{a}{2}} y^2<y$ if $y<(\epsilon+\delta)^{\frac{a}{2}}$, the above estimate can be further proceeded as
\begin{equation} \label{eq:energy_ineq}
  C't \ge \int_{C(\epsilon+\delta)}^{\min\{E(t), (\epsilon+\delta)^{\frac{a}{2}}\}} \frac{dy}{y} = \log\frac{\min\{ E(t), (\epsilon+\delta)^{\frac{a}{2}} \}}{C(\epsilon+\delta)}.
\end{equation}

The first consequence is that the existing time $T_{\epsilon,\delta}$, at which $E(T_{\epsilon,\delta})=\infty$ (because of Lemma \ref{lem:explosion}), must satisfy (here $C'$ is generally different from $C'$ above)
\[
  T_{\epsilon,\delta} \ge C' \left[\left( \frac{a}{2}-1 \right) \log(\epsilon+\delta) - 1\right],
\]
and clearly we have $T_{\epsilon,\delta}\to\infty$ as $\epsilon+\delta\to 0$.
Second, for any finite $T<T_{\epsilon,\delta}$, if $\epsilon+\delta$ is so small such that
\[
  \log\frac{1}{C(\epsilon+\delta)^{1-\frac{a}{2}}} > C'T,
\]
then it is only possible from (\ref{eq:energy_ineq}) that $E(t) \le C(\epsilon+\delta)e^{C't}$ for $t<T$; in other words, we show that
\[
  \sup_{0<t<T} E(t) \le C_T (\epsilon+\delta) \quad \text{for} \quad \epsilon+\delta \ll 1.
\]
This immediately gives the desired convergence result for $u^{\epsilon,\delta}$ in the theorem.
The boundedness of $u^{\epsilon,\delta}$ results from (\ref{eq:est_inf}) with $\gamma=1$ and the boundedness of $E$ as
\[
\begin{aligned}
  \|v\|_{L^\infty}^2 &\le C(\epsilon+\delta)^{-\frac{a}{2}}E + C\|v\|_{L^2}^2
  \le C(\epsilon+\delta)^{1-\frac{a}{2}} + C(\epsilon+\delta).
\end{aligned}
\]
Since $1-\frac{a}{2}>0$, the proof is completed.
\end{proof}

\begin{proof}[Proof of Theorem \ref{thm:uconvH1}]
  Define the same `residual' variable $W=(q,v,V)$ as in (\ref{eq:resvar}). This time the energy should be taken as
  \begin{equation} \label{eq:energyH1}
    E(t) = \|W(t,\cdot)\|_{H^1}^2 + (\epsilon+\delta)^a \|\Delta W(t,\cdot)\|_{L^2}^2
  \end{equation}
  for $a\ge 1$ in the $H^1$ framework. Then the assumptions in the theorem ensure that $E(0)\le C(\epsilon+\delta)$. In the energy estimate, we have a new term
  \[
    \frac{d}{dt} \|\nabla W\|_{L^2}^2 \le C\left( \sqrt\epsilon \|\nabla q\|_{L^2} + \sqrt\delta \|\nabla V\|_{L^2} + \|v\|_{H^1}^2 + \|v\|_{L^\infty}^2 \|\nabla v\|_{L^2}^2 \right),
  \]
  where $\|v\|_{H^1}^2$ and $\|v\|_{L^\infty}^2 \|\nabla v\|_{L^2}^2$ are the bounds of $\| \nabla(u^{NS} v) \|_{L^2}^2$ and $\|\nabla(v\otimes v)\|_{L^2}^2$, respectively.
  Using the estimates of $d\|W\|_{L^2}^2/dt$ and $d\|\nabla^2W\|_{L^2}^2/dt$ in the proof of Theorem \ref{thm:uconvL2}, we deduce that
  \[
    \frac{d}{dt}E(t) \le C\left( \sqrt{(\epsilon+\delta)E}+E + A \right)
  \]
  with
  \[
    A = \|v\|_{L^2}^2\|\nabla v\|_{L^2}^2 + \|v\|_{L^\infty}^2 \|\nabla v\|_{L^2}^2 + (\epsilon+\delta)^a \|\nabla^2 v\|_{L^2}^2 \left( \|v\|_{L^\infty}^2 + \|\nabla v\|_{L^2}^2 \right).
  \]
  Notice that for $\epsilon+\delta\ll 1$, each term in $A$ is bounded by $(\epsilon+\delta)^{-\frac{1}{2}}E(t)^2$ with $E(t)$ defined in (\ref{eq:energyH1}). Indeed, it is seen from (\ref{eq:est_inf}) and (\ref{eq:energyH1}) that
  \begin{equation} \label{eq:estvinf}
  \begin{split}
    \|v\|_{L^\infty}^2 &\le C\|v\|_{L^2}^\gamma \|\nabla v\|_{L^2}^{2-2\gamma} \|\nabla^2 v \|_{L^2}^\gamma + C\|v\|_{L^2}^2 \\
    &\le C E(t)^{\frac{\gamma}{2}+(1-\gamma)+\frac{\gamma}{2}}(\epsilon+\delta)^{-\frac{a\gamma}{2}} + C\|v\|_{L^2}^2 = C(\epsilon+\delta)^{-\frac{a\gamma}{2}}E(t) + C\|v\|_{L^2}^2.
  \end{split}
  \end{equation}
  It thus suffices to take $\gamma=a^{-1}$. 
  This shows the bounds for the terms containing $\|v\|_{L^\infty}^2$, whereas other terms are contained in $E(t)^2$.

  With $A$ such bounded, the energy estimate becomes
  \[
    \frac{d}{dt}E(t) \le C\left( \sqrt{(\epsilon+\delta)E}+E + (\epsilon+\delta)^{-\frac{1}{2}}E^2 \right).
  \]
  By applying the same estimation procedure as in the proof of Theorem \ref{thm:uconvL2} for (\ref{eq:enesL2}), we obtain $T_{\epsilon,\delta} \ge C' \left [-\frac{1}{2} \log(\epsilon+\delta) - 1\right ] \to \infty$ as $\epsilon+\delta\to 0$ and the bounded energy in a finite period of time $T<T_{\epsilon,\delta}$:
  \[
    \sup_{0<t<T} E(t) \le C_T(\epsilon+\delta) \quad \text{for} \quad \epsilon+\delta \ll 1.
  \]
  The convergence result for $u^{\epsilon,\delta}$ in the theorem is hence clear. The boundedness of $u^{\epsilon,\delta}$ results from (\ref{eq:estvinf}) and the boundedness of $E$ as
  \[
  \begin{aligned}
    \|v\|_{L^\infty}^2 &\le C (\epsilon+\delta)^{-\frac{a\gamma}{2}}E + C\|v\|_{L^2}^2
    \le C (\epsilon+\delta)^{1-\frac{a\gamma}{2}} + C(\epsilon+\delta).
  \end{aligned}
  \]
  Likewise, for any $a\ge 1$, we can take $\gamma\in(0,1]$ such that $1-\frac{a\gamma}{2}>0$, which gives immediately the desired boundedness of $u^{\epsilon,\delta}$ for $t\in(0,T)$. This completes the proof.
\end{proof}

\section{The proof of Theorem \ref{thm:pconv}} \label{sec:upconv}

This section is devoted to a proof of Theorem \ref{thm:pconv}. It is first noticed that the initial condition $(u_0^{\epsilon,\delta}, p_0^{\epsilon,\delta}, U_0^{\epsilon,\delta})$ satisfies all the assumptions in Theorem \ref{thm:uconvH1}, since for $a'=1+2a>1$, it holds that $(\epsilon+\delta)^{a'} \| \Delta u_0^{\epsilon,\delta}\|_{L^2}^2 \le C(\epsilon+\delta)$ and $(\epsilon+\delta)^{a'} \| \Delta \left( p_0^{\epsilon,\delta}, U_0^{\epsilon,\delta} \right) \|_{L^2}^2 \le C(\epsilon+\delta)\le C$ with sufficiently small values of $\epsilon$ and $\delta$.
Thus, for $T<T_{\epsilon,\delta}$, we have
\begin{equation} \label{eq:u_inf_control}
  \sup_{0<t<T} \| u^{\epsilon,\delta}(t,\cdot) \|_{L^\infty} \le C.
\end{equation}

While our focus (\ref{eq:relax}) for $(p^{\epsilon,\delta},u^{\epsilon,\delta},U^{\epsilon,\delta})$ is an approximation to (\ref{eq:ns}) for $(p^{NS},u^{NS},U^{NS})$, we construct an `intermediate' linear system for the unknowns $(p',u',U')$ as
\begin{subnumcases} {\label{eq:linear}}
  &$\epsilon \partial_t p' + \nabla \cdot u' = 0$, \label{eq:linear_a} \\
  &$\partial_t u' + \nabla \cdot U' + \nabla p' = 0$, \label{eq:linear_b} \\
  &$\delta \partial_t U' + \nabla u' = u^{NS} \otimes u^{NS} - U'$, \label{eq:linear_c} \\
  &$(p',u',U')(0,\cdot)=(p_0^{\epsilon,\delta},u_0^{\epsilon,\delta},U_0^{\epsilon,\delta})$.
\end{subnumcases}
Then our key strategy is to estimate $\| p^{\epsilon,\delta}-p'\|_{H^1}$ and $\| p'-p^{NS}\|_{H^1}$ separately so as to control $\| p^{\epsilon,\delta}-p^{NS} \|_{H^1}$, and the same has to be done for $\| u^{\epsilon,\delta}-u^{NS} \|_{H^1}$.

With such-prescribed initial conditions for (\ref{eq:linear}), we claim the following result analogous to (\ref{eq:u_inf_control}):
\begin{equation} \label{eq:u'_inf_control}
  \sup_{0<t<T} \| u'(t,\cdot) \|_{L^\infty} \le C.
\end{equation}
Indeed, the proof almost follows that of Theorem \ref{thm:uconvH1} by introducing $W=(q,v,V)$ as in (\ref{eq:resvar}). The governing equation for $V$ is now simply
\[\partial_t V + \frac{1}{\sqrt\delta}\nabla v = -\frac{V}{\delta}-\sqrt\delta\partial_t U^{NS}. \]
With the same energy $E(t) = \|W(t,\cdot)\|_{H^1}^2 + (\epsilon+\delta)^{a'} \|\Delta W(t,\cdot)\|_{L^2}^2$ as in (\ref{eq:energyH1}) (noting that $a'=1+2a>1$), we deduce that $dE/dt\le C\sqrt{(\epsilon+\delta)E}$ and hence $\sup_{0<t<T} E(t) \le C_T(\epsilon+\delta)$ for $\epsilon+\delta \ll 1$. Therefore, (\ref{eq:u'_inf_control}) results from (\ref{eq:est_inf}) in the same way as in the proof of Theorem \ref{thm:uconvH1}.

Now let us examine the differences between (\ref{eq:linear}) and the original system (\ref{eq:relax}). We have the following.

\begin{proposition} \label{prop:u'_u0}
  With all the assumptions in Theorem \ref{thm:pconv}, and further assuming that
  \begin{equation} \label{eq:u'_u0}
      \sup_{0<t<T}\|u'(t,\cdot)-u^{NS}(t,\cdot)\|_{H^1} \le C(\epsilon+\delta)
  \end{equation}
  for small values of $\epsilon$ and $\delta$, it then holds that
  \[
  \begin{aligned}
    \sup_{0<t<T} \|u^{\epsilon,\delta}(t,\cdot)-u'(t,\cdot)\|_{H^1} &\le C(\epsilon+\delta), \\
    \sup_{0<t<T} \sqrt\epsilon \| p^{\epsilon,\delta}(t,\cdot) - p'(t,\cdot) \|_{H^1} &\le C(\epsilon+\delta)
  \end{aligned}
  \]
  for small values of $\epsilon$ and $\delta$.
\end{proposition}

\begin{proof}
Denote
\[
  v = \frac{u^{\epsilon,\delta}-u'}{\mu}, \quad
  q = \sqrt\epsilon \frac{p^{\epsilon,\delta}-p'}{\mu}, \quad
  V = \sqrt\delta \frac{U^{\epsilon,\delta}-U'}{\mu},
\]
with $\mu=\max\{\sqrt\epsilon,\sqrt\delta\}$. We see from (\ref{eq:u_inf_control}) and (\ref{eq:u'_inf_control}) that
\begin{equation} \label{eq:vinf}
  \|v\|_{L^\infty} \le \frac{1}{\mu}\left( \|u^{\epsilon,\delta}\|_{L^\infty} + \| u' \|_{L^\infty} \right) \le \frac{C}{\mu}.
\end{equation}
Then, the system for $W=(q,v,V)$ reads as
\begin{equation}
\left\{ \
\begin{aligned}
  &\partial_t q + \frac{1}{\sqrt\epsilon} \nabla \cdot v = 0, \\
  &\partial_t v + \frac{1}{\sqrt\delta}\nabla\cdot V + \frac{1}{\sqrt\epsilon}\nabla q =0, \\
  &\partial_t V + \frac{1}{\sqrt\delta}\nabla v = \frac{\mu v\otimes v}{\sqrt\delta} + \frac{2u' v}{\sqrt\delta} + \frac{u'-u^{NS}}{\mu\sqrt\delta}\otimes u' + u^{NS}\otimes \frac{u'-u^{NS}}{\mu\sqrt\delta} - \frac{V}{\delta},
\end{aligned}
\right.
\end{equation}
with $u'v:=(u'\otimes v+v\otimes u')/2$. The initial condition is $W(0,\cdot)=0$.

Define the energy as
\[
  E(t) = \|W(t,\cdot)\|_{H^1}^2,
\]
and then the energy estimate is proceeded, for $0<t<T$, as
\[
\begin{aligned}
  \frac{d}{dt}\|W(t,\cdot)\|_{L^2}^2 &= \int_{\mathbb T^2} \frac{2V}{\sqrt\delta}:\left( \mu v\otimes v + 2u'v + \frac{u'-u^{NS}}{\mu}\otimes u' + u^{NS}\otimes\frac{u'-u^{NS}}{\mu} \right)dx - \frac{2}{\delta} \|V\|_{L^2}^2 \\
  &\le C \mu^2 \|v\otimes v\|_{L^2}^2 + C\|u'\|_{L^\infty}^2 \|v\|_{L^2}^2 + \frac{C}{\mu^2} \|u'-u^{NS}\|_{L^2}^2 \left(\|u'\|_{L^\infty}^2 + \|u^{NS}\|_{L^\infty}^2 \right) \\
  &\le C\left(\mu^2 \|v\|_{H^1}^4 + \|v\|_{L^2}^2 + \mu^2\right),
\end{aligned}
\]
where the third line is derived based on (\ref{eq:ladyzhen}) (for the first term) and (\ref{eq:u'_u0}).
Moreover, we treat
\[
  \frac{d}{dt}\|\nabla W(t,\cdot)\|_{L^2}^2 \le I_1 + I_2 + I_3 + I_4,
\]
term by term. Using H{\"o}lder's inequality and (\ref{eq:vinf}), we see that
\[
  I_1 = C\mu^2\|\nabla(v\otimes v)\|_{L^2}^2 \le C\mu^2 \|v\|_{L^\infty}^2 \|\nabla v\|_{L^2}^2
  \le C\|\nabla v\|_{L^2}^2.
\]
Then, we have
\[
\begin{aligned}
  I_2 = C\|\nabla(u'v)\|_{L^2}^2 &\le C\|\nabla (u'-u^{NS}) v\|_{L^2}^2 + C\|\nabla u^{NS}\|_{L^\infty}^2 \|v\|_{L^2}^2 + C \|u'\|_{L^\infty}^2 \|\nabla v\|_{L^2}^2 \\
  &\le C\|\nabla(u'-u^{NS})\|_{L^2}^2 \|v\|_{L^\infty}^2 + C \|v\|_{H^1}^2 \\
  &\le C\mu^2 + C\|v\|_{H^1}^2,
\end{aligned}
\]
\[
\begin{aligned}
  I_3 = \frac{C}{\mu^2}\| \nabla((u'-u^{NS})\otimes u') \|_{L^2}^2
  &\le \frac{C}{\mu^2} \| \nabla(u'-u^{NS})\otimes (u'-u^{NS}) \|_{L^2}^2 + \frac{C}{\mu^2} \| \nabla((u'-u^{NS})\otimes u^{NS}) \|_{L^2}^2 \\
  &\le \frac{C}{\mu^2} \|u'-u^{NS}\|_{L^\infty}^2 \| \nabla(u'-u^{NS})\|_{L^2}^2 +  \frac{C}{\mu^2}\| u'-u^{NS} \|_{H^1}^2 \\
  &\le C\mu^2,
\end{aligned}
\]
and, as a part of the estimate of $I_3$,
\[
\begin{aligned}
  I_4 = \frac{C}{\mu^2}\| \nabla(u^{NS}\otimes (u'-u^{NS})) \|_{L^2}^2 \le C\mu^2.
\end{aligned}
\]
As a result, we have
\[
  \frac{d}{dt}\|\nabla W(t,\cdot)\|_{L^2}^2 \le C\mu^2 + C\|v\|_{H^1}^2,
\]
whence
\[
  \frac{dE}{dt} \le C\left( \mu^2 + \|v\|_{H^1}^2 + \mu^2\|v\|_{H^1}^4 \right)
  \le C(\mu^2 + E + \mu^2 E^2)
\]
with $E(0)=0$. This leads to, for $0<t<T$, the estimate
\[
\begin{aligned}
  Ct \ge \int_0^{E(t)} \frac{dy}{\mu^2 y^2 + y + \mu^2}
  \ge \int_0^{\min\{E(t),\mu^{-2}\}} \frac{dy}{2y+\mu^2}
  = \frac{1}{2}\log\left( 1+\frac{2\min\{E(t),\mu^{-2}\}}{\mu^2} \right)
\end{aligned}
\]
(because $\mu^2y^2<y$ if $y<\mu^{-2}$) or equivalently,
\[
  \min\{E(t),\mu^{-2}\} \le C_T \mu^2.
\]
For sufficiently small $\mu<C_T^{-\frac{1}{4}}$, it is only possible that $E(t)\le C_T \mu^2$. This immediately leads to the desired estimation of $\|u^{\epsilon,\delta}-u'\|_{H^1}$ and $\|p^{\epsilon,\delta}-p'\|_{H^1}$.
\end{proof}

With Proposition \ref{prop:u'_u0}, it remains to verify both (\ref{eq:u'_u0}) and that
\begin{equation} \label{eq:p'_p0}
  \sup_{0<t<T} \sqrt\epsilon \| p'(t,\cdot) - p^{NS}(t,\cdot) \|_{H^1} \le C(\epsilon+\delta),
\end{equation}
which are proved in Sections 4.1 and 4.2, respectively.

\subsection{A proof of (\ref{eq:p'_p0})}
We start with the pressure. Recall that $\int_{\mathbb T^2} p^{NS}(t,x)dx=0$, and notice that
\[ \frac{d}{dt}\int_{\mathbb T^2} p'(t,x)dx = -\frac{1}{\epsilon} \int_{\mathbb T^2} \nabla\cdot u'dx=0. \]
Then we see from the Poincaré inequality that
\[
\begin{aligned}
  \|p'(t,\cdot)-p^{NS}(t,\cdot)\|_{L^2} &\le C \left\vert \int_{\mathbb T^2} p'(0,x)dx \right\vert + C\|\nabla p'(t,\cdot)-\nabla p^{NS}(t,\cdot)\|_{L^2} \\
  &\le C\|p'(0,\cdot)-p_0^{NS}\|_{L^2}+ C\|\nabla p'(t,\cdot)-\nabla p^{NS}(t,\cdot)\|_{L^2},
\end{aligned}
\]
where the second line is a result of the Cauchy-Schwarz inequality. Since $\sqrt\epsilon \|p'(0,\cdot)-p_0^{NS}\|_{H^1}= \sqrt\epsilon \|p_0^{\epsilon,\delta}-p_0^{NS}\|_{H^1} \le C(\epsilon+\delta)$, we only need to show that
\begin{equation} \label{eq:nabla_p'_p0}
  \sqrt\epsilon \|\nabla p'(t,\cdot)-\nabla p^{NS}(t,\cdot)\|_{L^2}\le C(\epsilon+\delta).
\end{equation}

\begin{proposition} \label{prop:prel}
  The pressure $p'$ in (\ref{eq:linear}) satisfies
  \[
    \epsilon\delta \partial_t^3 p' - (\epsilon+\delta) \partial_t \Delta p' + \epsilon \partial_t^2 p' - \Delta (p'-p^{NS}) = 0.
  \]
\end{proposition}

\begin{proof}
  Let $\dive^2 = \nabla\cdot(\nabla\cdot)$ act on (\ref{eq:linear_c}). Noticing that $\dive^2 \nabla u'=\Delta \dive{u'}$, we obtain
  \[
    \delta \partial_t \dive^2 U' + \Delta \dive u' = \dive^2 (u^{NS}\otimes u^{NS}) - \dive^2 U'.
  \]
  Using (\ref{eq:linear_b}) and $\dive^2 (u^{NS}\otimes u^{NS}) =-\Delta p^{NS}$ from (\ref{eq:ns}), we have
  \begin{equation} \label{eq:div2U}
    \delta \partial_t \dive^2 U' + \Delta \dive u' = -\Delta p^{NS} + \partial_t \dive u' + \Delta p'.
  \end{equation}
  Now let $\delta\partial_t\dive$ act on (\ref{eq:linear_b}). This leads to
  \[
  \begin{aligned}
    \delta \partial_t^2 \dive u' + \delta \partial_t \Delta p' &= -\delta \partial_t \dive^2 U' \\
    &= \Delta \dive u' + \Delta p^{NS} - \partial_t \dive u' - \Delta p',
  \end{aligned}
  \]
  where (\ref{eq:div2U}) is used to get the second expression. Then the assertion in the proposition follows immediately by replacing $\dive u'$ with $-\epsilon \partial_t p'$, thanks to (\ref{eq:linear_a}).
\end{proof}

Now we introduce
\begin{subequations}
\begin{align}
  f &= \delta \partial_t p' + p'-p^{NS}, \label{eq:fdef} \\
  g &= \sqrt{\epsilon \delta} \partial_t \nabla p',
\end{align}
\end{subequations}
and they satisfy
\begin{subequations} \label{eq:fg}
\begin{align}
  (\epsilon \partial_t^2 - \Delta) f - \sqrt{\frac{\epsilon}{\delta}} \nabla \cdot g &= \epsilon \partial_t^2 p^{NS}, \label{eq:fg_a} \\
  \partial_t g - \sqrt{\frac{\epsilon}{\delta}}\nabla\partial_t f &= -\frac{g}{\delta} + \sqrt{\frac{\epsilon}{\delta}} \partial_t \nabla p^{NS}, \label{eq:fg_b}
\end{align}
\end{subequations}
where (\ref{eq:fg_b}) is the restriction given by the definition of $f$ and $g$, and (\ref{eq:fg_a}) is the reformulation of Proposition \ref{prop:prel}.

The system (\ref{eq:fg}) is equipped with an energy
\begin{equation} \label{eq:E_fg}
  E(t) = \int_{\mathbb T^2} \left( \epsilon (\partial_t f)^2 + |\nabla f|^2 + |g|^2 \right) dx.
\end{equation}
But before we perform an energy estimate, let us explain the relationship between (\ref{eq:fg}) and our goal, the estimate of $\|\nabla (p'-p^{NS})\|_{H^1}$ in (\ref{eq:nabla_p'_p0}).
From (\ref{eq:fdef}) (the definition of $f$) an ordinary differential equation in time can be derived as (with spatial dependence omitted for clarity)
\[
  \partial_t Ap' + \frac{1}{\delta} Ap' = \frac{1}{\delta} Af + \frac{1}{\delta} Ap^{NS}
\]
for $A=\partial_t$ or $\nabla$. Solving this `first-order' equation for $Ap'(t)$ gives
\begin{equation} \label{eq:A_p'_p0}
  Ap'(t) = e^{-\frac{t}{\delta}} Ap'(0) + \int_0^t \frac{1}{\delta} e^{\frac{t'-t}{\delta}} Af(t')dt' + \int_0^t \frac{1}{\delta} e^{\frac{t'-t}{\delta}} Ap^{NS}(t')dt',
\end{equation}
or equivalently,
\[
\begin{aligned}
  A(p'-p^{NS})(t) =\ & e^{-\frac{t}{\delta}} \left[ A(p'-p^{NS})(0) + A(p^{NS}(0) - p^{NS}(t)) \right] \\
  & + \int_0^t \frac{1}{\delta} e^{\frac{t'-t}{\delta}} Af(t')dt'
  + \int_0^t \frac{1}{\delta} e^{\frac{t'-t}{\delta}} A(p^{NS}(t')-p^{NS}(t))dt'.
\end{aligned}
\]
Taking $A=\nabla$, we have the estimate
\begin{equation} \label{eq:nabla_p'_p0_est}
  \|\nabla(p'-p^{NS})(t,\cdot)\|_{L^2} \le C_t \|\nabla(p'-p^{NS})(0,\cdot)\|_{L^2} + \sup_{0<t'<t}\sqrt{E(t')} + C_t \delta,
\end{equation}
where $E(t)$ is defined in (\ref{eq:E_fg}), the second term of time integral is treated as
\[
  \| \int_0^t \frac{1}{\delta} e^{\frac{t'-t}{\delta}} \nabla f(t')dt' \|_{L^2} \le \sup_{0<t'<t} \|\nabla f(t',\cdot)\|_{L^2} \int_0^t \frac{1}{\delta} e^{\frac{t'-t}{\delta}}dt' \le \sup_{0<t'<t} \sqrt{E(t')},
\]
and the last term is deduced by the mean value theorem as
\[
  \| \int_0^t \frac{1}{\delta} e^{\frac{t'-t}{\delta}} \nabla(p^{NS}(t')-p^{NS}(t))dt' \|_{L^2} \le C_t \int_0^t \frac{1}{\delta} e^{-\frac{t-t'}{\delta}} (t-t')dt' \le C_t \delta.
\]
Since $\delta\le C\sqrt\epsilon$, thanks to (\ref{eq:nabla_p'_p0_est}), we only need to show that
\begin{equation} \label{eq:E_fg_est}
  \epsilon E(t) \le C_T(\epsilon+\delta)^2
\end{equation}
to verify (\ref{eq:nabla_p'_p0}) and hence (\ref{eq:p'_p0}).

To show (\ref{eq:E_fg_est}), the energy estimate for (\ref{eq:fg}) is proceeded as
\[
\begin{aligned}
  \frac{d}{dt}E(t) &= 2\epsilon \int_{\mathbb T^2} (\partial_t f) (\partial_t^2 p^{NS})dx - 2\int_{\mathbb T^2} \frac{|g|^2}{\delta}dx + 2\sqrt{\frac{\epsilon}{\delta}} \int_{\mathbb T^2} g \cdot \partial_t \nabla p^{NS} dx \\
  & \le C\epsilon \|\partial_t f\|_{L^2} + C\epsilon - \frac{3}{2\delta} \|g\|_{L^2}^2,
\end{aligned}
\]
where we used H{\"o}lder's inequality, in particular that $\sqrt{\frac{\epsilon}{\delta}} \int_{\mathbb T^2} g \cdot \partial_t \nabla p^{NS} dx \le \frac{1}{4\delta}\|g\|_{L^2}^2 + \epsilon \|\partial_t \nabla p^{NS}\|_{L^2}^2$. It further leads to
\[
  \frac{d}{dt}E(t) \le C\sqrt{\epsilon E} + C\epsilon,
\]
whence
\[
  Ct \ge \int_{E(0)}^E \frac{dy}{\sqrt{\epsilon y}+\epsilon}
  \ge \int_{\max\{E(0),\epsilon\}}^E \frac{dy}{2\sqrt{\epsilon y}} = \frac{1}{\sqrt\epsilon} \left( \sqrt E - \sqrt{\max\{E(0),\epsilon\}} \right)
\]
because of $\sqrt{\epsilon y}>\epsilon$ when $y>\epsilon$. Hence, we have
\[
\begin{aligned}
  E(t) &\le \left( C\sqrt\epsilon t+ \sqrt{\max\{E(0),\epsilon\}} \right)^2 \\
  &\le E(0)+C\epsilon(1+t^2).
\end{aligned}
\]
Clearly, (\ref{eq:E_fg_est}) holds if we can show that
\begin{equation} \label{eq:E0_fg_est}
  \epsilon E(0) \le C(\epsilon+\delta)^2.
\end{equation}

For this purpose, we deduce the initial conditions of the system (\ref{eq:fg}) as
\[
\begin{aligned}
  f(0,\cdot) &= -\frac{\delta}{\epsilon}\nabla \cdot u'(0,\cdot) + p'(0,\cdot) - p^{NS}(0,\cdot), \\
  \partial_t f(0,\cdot) &= \frac{\delta}{\epsilon}\left( \dive^2 U'(0,\cdot) + \Delta p'(0,\cdot) \right) - \frac{1}{\epsilon} \nabla \cdot u'(0,\cdot) - \partial_t p^{NS}(0,\cdot), \\
  g(0,\cdot) &= -\sqrt{\frac{\delta}{\epsilon}} \nabla \dive u'(0,\cdot),
\end{aligned}
\]
where we used $\partial_t p'=-\frac{1}{\epsilon}\dive u'$ and $\partial_t^2 p' = \frac{1}{\epsilon} \dive^2 U' + \Delta p'$ from (\ref{eq:linear_a}) and (\ref{eq:linear_b}). Then we examine $\epsilon E(0)$ term by term.
For the first term, we see that $\epsilon^2\|\partial_tp^{NS}(0,\cdot)\|_{L^2}^2 \le C\epsilon^2\le C(\epsilon+\delta)^2$,
\[
  \delta^2 \left(\|\dive^2 U'(0,\cdot)\|_{L^2}^2 + \|\Delta p'(0,\cdot)\|_{L^2}^2 \right)
  \le \delta^2 \left(\|\nabla\dive U'(0,\cdot)\|_{L^2}^2 + \|\Delta p'(0,\cdot)\|_{L^2}^2 \right) \le C(\epsilon+\delta)^2
\]
due to the assumptions in Theorem \ref{thm:pconv}, and
\begin{equation} \label{eq:divv_initial}
  \|\nabla\cdot u'(0,\cdot)\|_{L^2} \le C \| \nabla \dive u'(0,\cdot) \|_{L^2} \le C(\epsilon+\delta)
\end{equation}
due to the Poincaré inequality (on $\mathbb T^2$) and the assumption in Theorem \ref{thm:pconv}, giving
\[ \epsilon \int_{\mathbb T^2} \epsilon(\partial_t f(0,x))^2dx \le C(\epsilon+\delta)^2. \]
For the remaining terms, we have
\[
\begin{aligned}
  \epsilon\int_{\mathbb T^2} |\nabla f(0,x)|^2 + |g(0,x)|^2 dx
  &\le \left( \frac{\delta^2}{\epsilon} + \delta \right) \| \nabla \dive u'(0,\cdot) \|_{L^2}^2 + \epsilon \| \nabla p'(0,\cdot) - \nabla p_0^{NS} \|_{L^2}^2 \\
  &\le C(\epsilon+\delta)^2.
\end{aligned}
\]
Therefore, (\ref{eq:E0_fg_est}) is verified and hence the proof of (\ref{eq:p'_p0}) is complete.

\subsection{A proof of (\ref{eq:u'_u0})}
Noticing that
\[
  \frac{d}{dt}\int_{\mathbb T^2} u(t,x)dx = - \int_{\mathbb T^2} (\nabla\cdot U + \nabla p)dx = 0
\]
for $u=u'$ and $u^{NS}$, we then see from the Poincaré inequality that
\begin{equation} \label{eq:u'_u0_tp}
\begin{split}
  \|u'(t,\cdot)-u^{NS}(t,\cdot)\|_{L^2} &\le C \left\vert \int_{\mathbb T^2} \left( u'(0,x) - u_0^{NS} \right) dx \right\vert + C\|\nabla (u'-u^{NS})(t,\cdot)\|_{L^2} \\
  &\le C\|u'(0,\cdot)-u_0^{NS}\|_{L^2}+ C\|\nabla (u'-u^{NS})(t,\cdot)\|_{L^2}.
\end{split}
\end{equation}
Since $\|u'(0,\cdot)-u_0^{NS}\|_{L^2} \le C(\epsilon+\delta)$, to control $\|u'(t,\cdot)-u^{NS}(t,\cdot)\|_{H^1}$ in (\ref{eq:u'_u0}), it suffices to show that
\begin{equation}
  \| \nabla(u'-u^{NS})(t,\cdot)\|_{L^2} \le C(\epsilon+\delta).
\end{equation}
For this purpose, we apply Lemma \ref{lem:helmholtz} to control the $L^2$ norms of $\dive u'$ (noticing $\dive u^{NS}=0$) and $\nabla\times (u'-u^{NS})$, respectively.

The estimation of $\|\dive u'\|_{L^2}$ results from (\ref{eq:A_p'_p0}) (with $A=\partial_t$) as
\begin{equation} \label{eq:divuprimeest}
\begin{split}
  \|\nabla \cdot u'(t,\cdot) \|_{L^2} &= \epsilon \|\partial_t p'(t,\cdot)\|_{L^2} \\
  &\le \| \nabla \cdot u'(0,\cdot)\|_{L^2} + \sup_{0<t'<t} \sqrt{\epsilon E(t')} + C_t\epsilon \\
  &\le C(\epsilon+\delta).
\end{split}
\end{equation}
Here the first line results from (\ref{eq:linear_a}), the second line is derived in a similar way to that of (\ref{eq:nabla_p'_p0_est}), and the last control is due to (\ref{eq:E_fg_est}) and (\ref{eq:divv_initial}).

Then our final task is show that
\begin{equation} \label{eq:curl_u'_u0}
  \|\nabla\times (u'-u^{NS})(t,\cdot)\|_{L^2} \le C(\epsilon+\delta),
\end{equation}
which is known to hold at $t=0$ because of the theorem assumption and (\ref{eq:divv_initial}).

For this purpose, we denote the vorticity of (\ref{eq:linear}) by $\omega' = \nabla \times u'=\partial_{x_1}u_2'-\partial_{x_2}u_1'\in\mathbb R$ and $\Omega' = \nabla \times U'=(\partial_{x_1} U_{12}'-\partial_{x_2}U_{11}', \  \partial_{x_1}U_{22}'-\partial_{x_2}U_{21}') \in\mathbb R^2$.
The governing equation for $\omega'$ and $\Omega'$ reads as
\[
\left\{ \
\begin{aligned}
  &\partial_t \omega' + \nabla \cdot \Omega' =0, \\
  &\delta \partial_t \Omega' + \nabla \omega' = \nabla\times (u^{NS}\otimes u^{NS}) - \Omega'.
\end{aligned}
\right.
\]
Similarly, the vorticity $\omega^{NS}=\nabla\times u^{NS}$ and $\Omega^{NS} = \nabla \times U^{NS}$ of (\ref{eq:ns}) are governed by
\[
\left\{ \
\begin{aligned}
  &\partial_t \omega^{NS} + \nabla \cdot \Omega^{NS} =0, \\
  &\Omega^{NS} = -\nabla \omega^{NS} + \nabla \times (u^{NS}\otimes u^{NS}).
\end{aligned}
\right.
\]

Denote
\[
  \xi = \frac{\omega'-\omega^{NS}}{\sqrt\delta}, \quad
  X = \Omega' - \Omega^{NS} - e^{-\frac{t}{\delta}} (\Omega'(0,\cdot)-\Omega^{NS}(0,\cdot)).
\]
We then derive the governing system for $(\xi,X)$ as
\begin{equation}
\left\{ \
\begin{aligned}
  &\partial_t \xi + \frac{1}{\sqrt{\delta}} \nabla \cdot X = -\frac{1}{\sqrt\delta} e^{-\frac{t}{\delta}}\nabla\cdot(\Omega'(0,\cdot)-\Omega^{NS}(0,\cdot)), \\
  &\partial_t X + \frac{1}{\sqrt\delta} \nabla \xi = -\frac{X}{\delta} - \partial_t \Omega^{NS},
\end{aligned}
\right.
\end{equation}
with the initial condition satisfying $X(0,\cdot)=0$ and
\[
  \|\xi(0,\cdot)\|_{L^2} = \frac{1}{\sqrt\delta} \| (\omega'-\omega{NS})(0,\cdot) \|_{L^2} \le \frac{C(\epsilon+\delta)}{\sqrt\delta}.
\]
The energy estimate for $\|(\xi,X)\|_{L^2}^2$ gives:
\[
\begin{aligned}
  \frac{d}{dt}\|(\xi,X)\|_{L^2}^2 &= -\frac{2}{\sqrt\delta} e^{-\frac{t}{\delta}} \int_{\mathbb T^2} \xi \nabla\cdot(\Omega'-\Omega^{NS})(0,x)dx - 2\int_{\mathbb T^2} X\cdot(\partial_t\Omega^{NS}) dx - \frac{2}{\delta} \int_{\mathbb T^2} |X|^2 dx \\
  &\le \frac{2}{\sqrt\delta} e^{-\frac{t}{\delta}} \|\xi\|_{L^2}\cdot \| \nabla\cdot(\Omega'-\Omega^{NS})(0,\cdot) \|_{L^2} + C\delta,
\end{aligned}
\]
where H{\"o}lder's inequality $\|X\cdot(\delta_t\Omega^{NS})\|_{L^1} \le \frac{1}{2\delta}\|X\|_{L^2} + 2\delta \|\partial_t\Omega^{NS}\|_{L^2}$ was used.

Further denote
\[ h=\left( \|(\xi,X)\|_{L^2}^2+\delta \right)^{\frac{1}{2}}. \]
We see that
\[
\begin{aligned}
  \frac{dh}{dt}=\frac{1}{2h}\frac{d}{dt}\|(\xi,X)\|_{L^2}^2
  &\le \frac{1}{\sqrt\delta} e^{-\frac{t}{\delta}} \frac{\|\xi\|_{L^2}}{h}\cdot \| \nabla\cdot(\Omega'-\Omega^{NS})(0,\cdot) \|_{L^2} + C\frac{\delta}{h} \\
  &\le \frac{1}{\sqrt\delta} e^{-\frac{t}{\delta}} \| \nabla\cdot(\Omega'-\Omega^{NS})(0,\cdot) \|_{L^2} + Ch.
\end{aligned}
\]
It can be verified through a direct calculation (with integration by parts) that
\[ \|\dive (\Omega'-\Omega^{NS})(0,\cdot)\|_{L^2} \le \| \nabla\dive (U'-U^{NS})(0,\cdot) \|_{L^2}. \]
Moreover, we have
\[
\begin{aligned}
  & \delta \|\nabla \dive U'(0,\cdot)\|_{L^2} \le C(\epsilon+\delta), \\
  &\delta \|\nabla \dive U^{NS}(0,\cdot)\|_{L^2} \le C\delta\le C(\epsilon+\delta),
\end{aligned}
\]
which lead to
\[
  \frac{dh}{dt}\le Ch + \frac{C(\epsilon+\delta)}{\delta^{\frac{3}{2}}}e^{-\frac{t}{\delta}}
\]
with
\[ h(0)\le \frac{C(\epsilon+\delta)}{\sqrt\delta}. \]
We thus obtain
\[
  h(t) \le \frac{e^{Ct}}{1+C\delta} \left( 1-e^{-\left( C+\frac{1}{\delta} \right)t} \right) \frac{C(\epsilon+\delta)}{\sqrt\delta}
  \le \frac{C_T(\epsilon+\delta)}{\sqrt\delta},
\]
and (\ref{eq:curl_u'_u0}) follows immediately for $0<t<T$:
\[ \| (\omega'-\omega^{NS})(t,\cdot) \|_{L^2} = \sqrt\delta \|\xi(t,\cdot)\|_{L^2} \le \sqrt\delta h(t) \le C(\epsilon+\delta). \]
This completes the proof of Theorem \ref{thm:pconv}.

\begin{remark}
  In the proof of (\ref{eq:u'_u0}), we use the fact that $\frac{d}{dt}\int u'dx=0$ on $\mathbb T^2$, which may not hold on other types of domains with boundaries. In these situations, proper boundary conditions are needed.
\end{remark}

\begin{remark}
  The use of an intermediate system plays an essential role in the convergence proof of $p^{\epsilon,\delta}$. For our current regime $\delta = O(\sqrt\epsilon)$, another choice of the intermediate system is to take $\delta=0$ in (\ref{eq:linear}). In this way, a result similar to Theorem \ref{thm:pconv} can be derived, which is omitted here. But it does not apply for $\delta \gtrsim \sqrt\epsilon$ either.
\end{remark}

\begin{remark} \label{rm:fric}
  Finally, for the model with a linear friction term $-u$, Theorem \ref{thm:pconv} can be proved with the same arguments as in this section. Notice that in this case, we have
  \[ \frac{d}{dt}\int_{\mathbb T^2} (u'-u^{NS}) dx = - \int_{\mathbb T^2} (u'-u^{NS}) dx, \]
  and thus $\left | \int_{\mathbb T^2} (u'-u^{NS}) dx \right |$ decreases in time. Therefore, (\ref{eq:u'_u0_tp}) still holds in this case.
\end{remark}

\section{Conclusions} \label{sec:concl}
In this paper, we introduce a novel hyperbolic relaxation approximation (\ref{eq:relax}) to the incompressible Navier-Stokes equation (\ref{eq:ns}) with the artificial compressibility (AC) method. The system contains two relaxation parameters $\epsilon$ (for the AC term) and $\delta$ (for the first-order relaxation term), and the solution from smooth initial data exists for each $(\epsilon,\delta)$. We show that, as $\epsilon$ and $\delta$ both tend to zero, the existing time of the smooth solution of (\ref{eq:relax}) goes to infinity, and these solutions (namely, fluid velocity and pressure) converge to solutions of (\ref{eq:ns}) for $\delta=o(\sqrt\epsilon)$ and properly-prepared initial data. 
Our analysis consists of (\textit{i}) construction of an `intermediate' linear system (\ref{eq:linear}) and (\textit{ii}) energy estimates of the errors between (\ref{eq:linear}) and the two target systems (\ref{eq:ns}) and (\ref{eq:relax}). In Step (\textit{ii}), the error of the fluid velocities is controlled respectively by their divergence and curl, the latter of which (namely the vorticity) further requires initial layer corrections.

Future work could be directed towards deeper analyses of the singular limits of (\ref{eq:relax}), including more general relations between the two parameters and three-dimensional cases. Numerical investigations of (\ref{eq:relax}) will also be highly beneficial.

\section*{Acknowledgements}
Funding by the Deutsche Forschungsgemeinschaft (DFG, German Research Foundation) - SPP 2410 \textit{Hyperbolic Balance Laws in Fluid Mechanics: Complexity, Scales, Randomness (CoScaRa)} and the Sino-German cooperation group on \textit{Advanced Numerical Methods for nonlinear Hyperbolic Balance Laws and their Applications} is gratefully acknowledged.


\appendix
\section{Proof of Lemma \ref{lem:GNineq}} \label{app:GNineq}

\begin{proof}[Proof of (\ref{eq:est_dv})]
  For $v\in H^2(\mathbb T^2)$, it holds that
  \[
  \begin{aligned}
    \|\nabla v\|_{L^2}^2 &= \sum_{k\in\mathbb Z^2} |k|^2 |\hat v(k)|^2 \\
    &\le \left( \sum_{k\in\mathbb Z^2} |\hat v(k)|^2 \right)^{\frac{1}{2}} \left( \sum_{k\in\mathbb Z^2} |k|^4 |\hat v(k)|^2 \right)^{\frac{1}{2}} \\
    &= \|v\|_{L^2} \|\nabla^2 v\|_{L^2}
  \end{aligned}
  \]
  with $\hat v(k)$ being the Fourier transform of $v$ on $\mathbb T^2$.
\end{proof}

\begin{proof}[Proof of (\ref{eq:est_inf})]
  Using the Fourier transform and H\"{o}lder's inequality, we have
  \[ \|v\|_{L^\infty} \le |\hat v(0)| + \sum_{\substack{k\in\mathbb Z^2 \\ k\ne 0}} |\hat v(k)| \le C\|v\|_{L^2} + \sum_{\substack{k\in\mathbb Z^2 \\ k\ne 0}} |\hat v(k)|. \]
  It then holds that
  \[
  \begin{aligned}
    \sum_{\substack{k\in\mathbb Z^2 \\ k\ne 0}} |\hat v(k)| &= \sum_{1\le |k|\le N} |\hat v(k)|^\gamma \cdot |k\hat v(k)|^{1-\gamma} \cdot |k|^{\gamma-1} + \sum_{|k|>N} |\hat v(k)|^{\frac{\gamma}{4}} \cdot |k\hat v(k)|^{1-\gamma} \cdot |k^2\hat v(k)|^{\frac{3\gamma}{4}} \cdot |k|^{-1-\frac{\gamma}{2}} \\
    &\le C \|\hat v\|_{l^2}^\gamma \cdot \|k\hat v\|_{l^2}^{1-\gamma} \left( \sum_{1\le|k|\le N} |k|^{2\gamma-2} \right)^{\frac{1}{2}}
    + C \|\hat v\|_{l^2}^{\frac{\gamma}{4}} \cdot \|k\hat v\|_{l^2}^{1-\gamma} \cdot \|k^2\hat v\|_{l^2}^{\frac{3\gamma}{4}} \left( \sum_{|k|>N} |k|^{-2-\gamma} \right)^{\frac{1}{2}} \\
    &\le C \|v\|_{L^2}^\gamma \cdot \|\nabla v\|_{L^2}^{1-\gamma} \cdot N^\gamma + C \|v\|_{L^2}^{\frac{\gamma}{4}} \cdot \|\nabla v\|_{L^2}^{1-\gamma} \cdot \|\nabla^2 v\|_{L^2}^{\frac{3\gamma}{4}} \cdot N^{-\frac{\gamma}{2}}
\end{aligned}
\]
for $0<\gamma\le 1$ and $N\in\mathbb N$. Here the second line results from the extended H\"{o}lder's inequality. The third line resorts to the elementary facts that
\[
\begin{aligned}
\int_{1}^{N}x^{2\gamma-2}dx &<\int_{1}^{N}x^{2\gamma-1}dx <\frac{1}{2\gamma} N^{2\gamma}, \\
\int_{N}^{\infty}x^{-2-\gamma}dx &<\int_{N}^{\infty}x^{-1-\gamma}dx=\frac{1}{\gamma}N^{-\gamma}, 
\end{aligned}
\]
so here the constants $C$ rely on $\gamma$.

Now take $N\in\mathbb N$ such that
\[
  N-1\le \left(\frac{\|\nabla^2 v\|_{L^2}}{\|v\|_{L^2}}\right)^{\frac{1}{2}} <N.
\]
If $N\ge 2$, it is straightforward to show that
\[ \left(\frac{\|\nabla^2 v\|_{L^2}}{\|v\|_{L^2}}\right)^{\frac{1}{2}} < N \le 2\left(\frac{\|\nabla^2 v\|_{L^2}}{\|v\|_{L^2}}\right)^{\frac{1}{2}}, \]
and the above estimate directly gives
\begin{equation} \label{eq:vinf_kne0}
  \sum_{\substack{k\in\mathbb Z^2 \\ k\ne 0}} |\hat v(k)| \le C\|v\|_{L^2}^{\frac{\gamma}{2}} \cdot \|\nabla v\|_{L^2}^{1-\gamma} \cdot\|\nabla^2 v\|_{L^2}^{\frac{\gamma}{2}}
\end{equation}
for the both terms as desired. On the other hand, if $N=1$, which implies $\|\nabla^2 v\|_{L^2}<\|v\|_{L^2}$, then the above estimate results in
\begin{equation} \label{eq:vinf_N1_1}
  \sum_{\substack{k\in\mathbb Z^2 \\ k\ne 0}} |\hat v(k)| \le C \|v\|_{L^2}^\gamma \cdot \|\nabla v\|_{L^2}^{1-\gamma}.
\end{equation}
Meanwhile, it holds that
\begin{equation} \label{eq:vinf_N1_2}
\begin{split}
  \sum_{\substack{k\in\mathbb Z^2 \\ k\ne 0}} |\hat v(k)| &= \sum_{|k|\ge 1} |k\hat v|^{1-\gamma} \cdot |k^2\hat v|^{\gamma} \cdot |k|^{-1-\gamma} \\
  & \le \|k\hat v\|_{l^2}^{1-\gamma} \cdot \|k^2 \hat v\|_{l^2}^\gamma \cdot \left( \sum_{|k|\ge 1} |k|^{-2-2\gamma} \right)^{\frac{1}{2}} \\
  & \le C \|\nabla v\|_{L^2}^{1-\gamma} \cdot \|\nabla^2 v\|_{L^2}^\gamma.
\end{split}
\end{equation}
Clearly, (\ref{eq:vinf_N1_1}) and (\ref{eq:vinf_N1_2}) imply (\ref{eq:vinf_kne0}) immediately for $N=1$. This completes the proof.
\end{proof}

\section{A proof of Lemma \ref{lem:explosion}} \label{app:blow}

This appendix is devoted to a proof of Lemma \ref{lem:explosion}. The proof will require one more interpolation inequality for $v\in H^2(\mathbb T^2)$:
\begin{equation} \label{eq:nabla_v_L4}
  \|\nabla v\|_{L^4}^2 \le C \|v\|_{L^\infty} \|\nabla^2 v\|_{L^2}.
\end{equation}
This can be proved through integration by parts and the H\"{o}lder's inequality:
\[
\begin{aligned}
  \|\nabla v\|_{L^4}^4 &= \int_{\mathbb T^2} \nabla v\cdot \nabla v |\nabla v|^2 dx \\
  &= - \int_{\mathbb T^2} v \nabla^2 v |\nabla v|^2 dx - \int_{\mathbb T^2} v\nabla v\cdot \nabla(|\nabla v|^2) \\
  & \le C \|v\|_{L^\infty} \|\nabla^2 v\|_{L^2} \|\nabla v\|_{L^4}^2.
\end{aligned}
\]

Assume for the sake of contradiction that
\begin{equation} \label{eq:contr}
  \int_0^T \|U(t)\|_{L^\infty} dt < \infty.
\end{equation}
Now we perform an energy estimate of (\ref{eq:quad_model}), yielding that
\[
\begin{aligned}
  \frac{d}{dt}\|U(t)\|_{H^2}^2 &= 2 \int_{\mathbb T^2} (U,\nabla U,\nabla^2 U) \cdot \left[ (F,\nabla F,\nabla^2 F) + (Q(U),\nabla Q(U), \nabla^2 Q(U)) \right] dx \\
  & \le C \|U(t)\|_{H^2} \cdot \|F(t)\|_{H^2} + C \|U(t)\|_{H^2} \left(\int_{\mathbb T^2} \left( |Q(U)|^2 + |\nabla Q(U)|^2 + |\nabla^2 Q(U)|^2 \right) dx \right)^{\frac{1}{2}}.
\end{aligned}
\]
Since $Q(U)$ only contains terms with the form $aU_i+bU_iU_j$, we have $\int |Q(U)|^2 dx \le C \|U(t)\|_{L^2}^2 \big(1+\|U(t)\|_{L^\infty}^2 \big)$, $\int |\nabla Q(U)|^2 dx \le C \|\nabla U(t)\|_{L^2}^2 \big(1+\|U(t)\|_{L^\infty}^2 \big)$, and
\[
\begin{aligned}
  \int |\nabla^2 Q(U)|^2 dx &\le C \|\nabla U\|_{L^4}^4 + C\|\nabla^2 U(t)\|_{L^2}^2 \big(1+\|U(t)\|_{L^\infty}^2 \big) \\
  &\le C\|\nabla^2 U(t)\|_{L^2}^2 \big(1+\|U(t)\|_{L^\infty}^2 \big),
\end{aligned}
\]
where we used (\ref{eq:nabla_v_L4}) for the last inequality. Thus, we obtain
\[
  \frac{d}{dt}\|U(t)\|_{H^2}^2 \le C \|U(t)\|_{H^2} \cdot \|F(t)\|_{H^2} + C \|U(t)\|_{H^2}^2 \left( 1+\|U(t)\|_{L^\infty} \right),
\]
or equivalently,
\[
  \frac{d}{dt} \left[ \|U(t)\|_{H^2} \exp \left( -C\int_0^t (1+\|U(s)\|_{L^\infty})ds \right) \right] \le C \|F(s)\|_{H^2} \exp \left( -C\int_0^t (1+\|U(s)\|_{L^\infty})ds \right)
\]
for $t<T$. By the Gronwall's inequality this leads to
\[
\begin{aligned}
  \|U(t)\|_{H^2} \le &\|U(0)\|_{H^2} \exp \left(C\int_0^t (1+\|U(s)\|_{L^\infty})ds \right) \\
  &+ C\int_0^t \|F(s)\|_{H^2} \exp \left(C\int_s^t (1+\|U(w)\|_{L^\infty})dw \right) ds.
\end{aligned}
\]
Considering (\ref{eq:contr}), we see from the above control of $\|U(t)\|_{H^2}$ and the dominated convergence theorem that
\[
  U\in L^\infty ([0,T),H^2(\mathbb T^2)).
\]
By (\ref{eq:quad_model}), using (\ref{eq:ladyzhen}) and (\ref{eq:est_inf}), we get
\[
  \|\partial_t U(t)\|_{H^1} \le C\left( \|U(t)\|_{H^2}+\|U(t)\|_{H^2}^2 + \|F(t)\|_{H^1} \right),
\]
indicating that
\[
  \partial_t U \in L^\infty([0,T),H^1(\mathbb T^2))
\]
and hence
\[
  U\in C^{0,1}([0,T),H^1(\mathbb T^2)).
\]
Thus, $U$ can be continuously extended to $t = T$ in $H^1(\mathbb T^2)$.

Now $U(t,x)$ with $t < T$ is a solution of the linear equation
\[
  \partial_t V+\sum_{i=1}^2 A_i \partial_{x_i} V = q(t,x)V + F(t,x)
\]
with
\[
  q \in L^\infty([0,T],H^2(\mathbb T^2)) \cap C^{0,1}([0,T],H^1(\mathbb T^2)).
\]
Thanks to Theorem 1 in \cite{kato1975}, the solution of this linear system exists uniquely in $V \in C^0([0,T],H^2(\mathbb T^2)) \cap C^1([0,T],H^1(\mathbb T^2))$. Consequently, $U=V$ on $t\in[0,T)$ and can be continuously extended to $t = T$ in $H^2(\mathbb T^2)$.

Finally, using the local existence theory in \cite{kato1975}, we solve
\[
  \partial_t U' + \sum_{i=1}^{2} A_i \partial_{x_i} U' = Q(U') + F(t,x)
\]
for $T<t<T+\epsilon$ with $\epsilon>0$ and $U'(T)=U(T)$. Then we see that
\[
  U''(t): = \left\{
  \begin{aligned}
    U(t), \quad &t\le T \\
    U'(t), \quad &T<t<T+\epsilon
  \end{aligned}
  \right. \in C^0([0,T+\epsilon),H^2(\mathbb T^2)) \cap C^1([0,T+\epsilon),H^1(\mathbb T^2))
\]
solves (\ref{eq:quad_model}) on $[0,T+\epsilon)$, which contradicts the fact that $T$ is the maximum existing time of the solution. This completes the proof of Lemma \ref{lem:explosion}.

\bibliographystyle{amsplain}
\bibliography{references}

\end{document}